\documentclass[11pt]{article} 

\usepackage{amsmath,amssymb}
\usepackage{amsthm}
\usepackage{color}
\usepackage[latin1]{inputenc}
\usepackage[T1]{fontenc}
\usepackage[english]{babel}
\usepackage{mathtools}
\usepackage{amsfonts}
\usepackage{graphicx}
\usepackage{dsfont}
\usepackage{epsfig}
\usepackage{epstopdf}
\usepackage{ae}
\usepackage{amsthm}
\usepackage{thmtools}
\usepackage{endnotes}
\usepackage{setspace}
\usepackage{appendix}
\usepackage{float}
\usepackage{subfigure}
\usepackage{multirow}
\usepackage{chngpage} 
\usepackage{bbm}

\usepackage{titlesec}
\titleformat*{\section}{\large \bfseries}
\titlespacing{\section}{0pt}{\parskip}{\parskip}

\titleformat*{\subsection}{\normalsize \bfseries}
\titlespacing{\subsection}{0pt}{\parskip}{\parskip}

\usepackage{setspace}

\numberwithin{equation}{section}

\newtheorem{theorem}{Theorem}[section]

\newtheorem{assumption}{Assumption}[section]

\title{\Large{Integral equation characterization of the Feynman-Kac formula for  a regime-switching diffusion}} 

\author{
Adriana Ocejo\thanks{Department of Mathematics and Statistics, UNC Charlotte. University City Blvd., Charlotte, NC, 28262, USA.
e-mail: amonge2@uncc.edu, phone: +1(704)687-1413.}}

\begin{document}

\maketitle

\begin{abstract}
In this paper, we provide an integral equation characterization of the solution to a Cauchy problem
associated to the Feynman-Kac formula for a regime-switching diffusion.
We give a sufficient condition to guarantee the uniqueness of solutions to the integral equation and
provide an example in the context of option pricing under the Ornstein-Uhlenbeck regime-switching model.
\end{abstract}

{\bf Key words.} Cauchy problem; Feynman-Kac; regime-switching; integral equation; contraction mapping.


\section{Introduction}


In this paper, we derive an integral equation characterization of the solution to the Cauchy problem
 \begin{equation} \label{eq:Cauchy_p}
 \begin{split}
\mathcal{L}v(t,x,i)-r(i)v(t,x,i) &=0, \;\quad \qquad  \mbox{in $[0,T)\times \mathbb{R}\times \mathcal{M}$,} \\
v(T,x,i) &=\varphi(x,i),  \quad \mbox{in $\mathbb{R}\times \mathcal{M}$,}
 \end{split}
 \end{equation}
where $\mathcal{M}:=\{1,2,\ldots,m\}$, $r(i)\geq 0$, $\varphi(\cdot,i)$ is continuous in $\mathbb{R}$ and of polynomial growth, 
and for any $f(\cdot,\cdot,i)\in C^{1,2}([0,T]\times \mathbb{R})$ the operator $\mathcal{L}$ is defined by
\[
\mathcal{L}f(t,x,i)=\frac{\partial f(t,x,i)}{\partial t}+b(x,i)\frac{\partial f(t,x,i)}{\partial x}+\frac{1}{2}\sigma^2(x,i)\frac{\partial^2 f(t,x,i)}{\partial x^2}+Qf(t,x,\cdot)(i).
\]
Here, $Q=[q_{ij}]_{m\times m}$ represents the generator of a continuous-time Markov chain in the finite state space $\mathcal{M}$ with transition rates $q_{ij}\geq 0$ for $i\neq j$, and
\[
Qf(t,x,\cdot)(i)=\sum_{j\neq i} q_{ij}[f(t,x,j)-f(t,x,i)],
\qquad
\sum_{j\neq i} q_{i j}=-q_{ii}=: q_i.
\]
The solution to \eqref{eq:Cauchy_p} is known to have a stochastic representation, under suitable conditions on $b$ and $\sigma$ recalled below, via the Feynman-Kac formula (see Theorem 6 in \cite{Baranetal}):
\begin{equation} \label{eq:cond_exp_general}
v(t,x,i)=\mathbb{E}_{t,x,i}[e^{-\int_t^T r(\alpha_u)du}\varphi(X_T,\alpha_T)], \qquad (t,x,i)\in [0,T)\times \mathbb{R}\times \mathcal{M},
\end{equation}
where we write $\mathbb{E}_{t,x,i}[\cdot]=\mathbb{E}[ \,\cdot  \mid X_t=x,\alpha_t=i]$ for short,
and $(X_t,\alpha_t)$ is
is a regime-switching diffusion (see e.g. \cite{Yin&Zhu2010}) defined on a filtered probability space $(\Omega,\mathcal{F}, \{\mathcal{F}_t\}_{t\geq 0}, \mathbb{P})$.

Conditional expectations of the form in \eqref{eq:cond_exp_general} arise in many applications, particularly in the pricing of financial contracts (see e.g. \cite{Musiela&Rutkowski} and references therein).
This expectation can be approximated by numerically solving the coupled PDE in \eqref{eq:Cauchy_p} by finite-difference methods (\cite{Chen&Insley}, \cite{Duffy}, \cite{Li2016}),
by trinomial tree methods (\cite{Ma&Zhu}), or by Monte Carlo simulation (\cite{Glasserman}).
In this paper, we propose an analytical representation of $v$ in \eqref{eq:cond_exp_general} as the fixed point of an integral equation where we exploit the contraction theorem on Banach spaces.
We give a sufficient condition for this representation to hold and provide an
example in the context of commodity derivatives.


\section{Main result}

The regime-switching process $(X,\alpha)=\{(X_t,\alpha_t)\}_{t\geq 0}$ is described by
\[
dX_t=b(X_t,\alpha_t)dt+\sigma(X_t,\alpha_t)dW_t,
\]
and
\[
\mathbb{P}( \alpha(t+\Delta)=j \mid \alpha(t)=i, X_s, \alpha(s), s\leq t)=q_{ij} \Delta +o(\Delta), \qquad i\neq j,
\]
where $W=(W_t)_{t\geq 0}$ is a standard Brownian motion independent of the continuous-time Markov chain $\alpha$,
and
the functions $b(\cdot,\cdot):\mathbb{R}\times \mathcal{M}\mapsto \mathbb{R}$ and $\sigma(\cdot,\cdot):\mathbb{R}\times \mathcal{M}\mapsto \mathbb{R}$
satisfy the linear growth and local Lipschitz conditions, respectively:
\[
|b(x,i)|+|\sigma(x,i)|  \leq K(1+|x|), \qquad i\in \mathcal{M},
\]
for some constant $K>0$, and for each $N\geq 1$, there exists a positive constant $M_N$ such that for all $i\in \mathcal{M}$, and $x,y\in \mathbb{R}$ with $|x|\vee|y|\leq M_N$,
\[
|b(x,i)-b(y,i)|\vee |\sigma(x,i)-\sigma(y,i)|\leq M_N|x-y|.
\]

For each $i\in \mathcal{M}$, we shall denote by $X^{(i)}_u$, $u\geq t$, the solution to
the non-regime switching stochastic differential equation $dX_u=b(X_u,i)du+\sigma(X_u,i)dW_u$ and $X_t=x$.
Consider the Fokker-Planck equation associated with the transition density $f_i(u,y;t,x)$ of the random variable $X^{(i)}_u$, $u\geq t$:
\begin{equation}\label{eq:F-P}
\frac{\partial}{\partial u}f_i(u,y;t,x)=-\frac{\partial}{\partial y}[b(y,i) f_i(u,y;t,x)]+\frac{1}{2}\frac{\partial^2}{\partial y^2}[\sigma^2(y,i)f_i(u,y;t,x)].
\end{equation}

\pagebreak
Let $\mathcal{S}$ denote the Banach space of all bounded measurable functions $h:\mathcal{E}\mapsto \mathbb{R}$, where $\mathcal{E}:=[0,T]\times \mathbb{R}\times \mathcal{M}$,
with the supremum norm
\[
||h||:=\sup_{(t,x,i)\in \mathcal{E}}|h(t,x,i)|.
\]

\begin{assumption} \label{assump:boundedness}\normalfont
There exist a continuous function $D:[0,T]\times \mathbb{R}\mapsto \mathbb{R}\backslash\{0\}$
such that for each $i\in \mathcal{M}$, the process $\{e^{-r(i)t}D(t,X_t^{(i)})\}_{0\leq t\leq T}$ is a supermartingale and
the function $H(t,x,i):=v(t,x,i)/D(t,x)$ belongs to $\mathcal{S}$.
\end{assumption}

The function $D$ may be understood as a type of \textit{dampening factor} of the value function $v$ that forces it to be bounded.
\textit{Dampening} (or discounting) $v(t,x,i)$ by $D(t,x)$ allows obtaining a contractive condition for the application of the fixed point theorem.
Note that if $\varphi$ is bounded then Assumption \ref{assump:boundedness} automatically holds by choosing $D(t,x)\equiv 1$.
We give a non-trivial example in the next section when $\varphi$ is not bounded.

\begin{theorem} \label{thm:main}
Let Assumption \ref{assump:boundedness} hold. Then $H(t,x,i):=v(t,x,i)/D(t,x)$ is the unique solution of the integral equation
\begin{equation} \label{eq:inteq}
h(t,x,i)=H_0(t,x,i)+\mathcal{T}(h)(t,x,i)
\end{equation}
where
\begin{equation} \label{eq:inteq_initial}
H_0(t,x,i)\equiv e^{-q_i(T-t)} \,\mathbb{E}_{t,x,i}\left[e^{- r(i)(T-t)}\varphi(X^{(i)}_T,i) \right]/D(t,x)
\end{equation}
and the operator $\mathcal{T}:\mathcal{S}\mapsto \mathcal{S}$ is a contraction on $\mathcal{S}$ defined as
\[
\mathcal{T}(h)(t,x,i)=D^{-1}(t,x)\sum_{j\neq i} q_{ij}\int_t^T \int_\mathbb{R} \, e^{-(q_i+r(i))(u-t)} D(u,y) h(u,y,j)\, f_i(u, y;t,x)dy du.
\]
\end{theorem}
\begin{proof}
Fix $t\in [0,T)$ and consider the first jump time of the Markov chain $\alpha$ after time $t$, $\tau(t):=\inf\{u\geq t: \alpha_u\neq \alpha_t\}$.
Suppose that $X_t=x$, $\alpha_t=i$.

We first show that $H$ is indeed a solution to \eqref{eq:inteq}.
Using that $\tau(t)\sim Exp(q_i)$, by the law of total expectation and properties of the conditional expectation with respect to an event, we can split $v$ as follows:
\begin{equation} \label{eq:firstsplit}
\begin{aligned}
v(t,x,i)   &=\mathbb{E}_{t,x,i}\left[\, e^{-\int_t^T r(\alpha_u) du}\varphi(X_T,\alpha_T)\mid \tau(t)> T\,\right]\mathbb{P}(\tau(t)>T) \\ 
            & \qquad +\mathbb{E}_{t,x,i}\left[\, e^{-\int_t^T r(\alpha_u) du}\varphi(X_T,\alpha_T)\,\mathbbm{1}(\tau(t)\leq T)\,\right]\\ 
             &=e^{-q_i(T-t)}\mathbb{E}_{t,x,i}\left[e^{- r(i)(T-t)}\varphi(X^{(i)}_T,i) \right] \\
             & \qquad +\mathbb{E}_{t,x,i}\left[\, e^{-\int_t^T r(\alpha_u) du}\varphi(X_T,\alpha_T)\,\mathbbm{1}(\tau(t)\leq T)\,\right].
\end{aligned}
\end{equation}
We can further write the second expectation on the right hand side of \eqref{eq:firstsplit} as
\[
\begin{aligned}
\mathbb{E}_{t,x,i} &\left[\, e^{-\int_t^T r(\alpha_u) du}\varphi(X_T,\alpha_T)\,\mathbbm{1}(\tau(t)\leq T)\,\right]  \\
& \hspace{-1cm} = \int_t^T q_i \,e^{-q_i(u-t)} \mathbb{E}_{t,x,i} \left[\, e^{-\int_t^T r(\alpha_s) ds}\varphi(X_T,\alpha_T)\,\mid \tau(t)=u\right]  du \\
& \hspace{-1cm} = \int_t^T q_i \,e^{-(q_i+r(i))(u-t)} \mathbb{E}_{t,x,i}
\left[\, \mathbb{E}\left[\, e^{-\int_u^T r(\alpha_s) ds}\varphi(X_T,\alpha_T) \mid \mathcal{F}_u \right] \mid \tau(t)=u \right] du \\
& \hspace{-1cm} = \int_t^T q_i \,e^{-(q_i+r(i))(u-t)} \mathbb{E}_{t,x,i} \left[\, v(u,X_u,\alpha_u)\mid \tau(t)=u\right]du \\
& \hspace{-1cm} = \sum_{j\neq i} q_{ij}\int_t^T  \int_\mathbb{R} \,e^{-(q_i+r(i))(u-t)} \, D(u,y)\,H(u,y,j)f_i(u,y; t,x)\,dy\,du,
\end{aligned}
\]
where we used the Markov property of the two-component process $(X,\alpha)$.
Multiply both sides of \eqref{eq:firstsplit} by $D^{-1}(t,x)$ to obtain that $H$ indeed solves \eqref{eq:inteq}.

Now, let us show that the operator $\mathcal{T}$ is a contraction mapping on $\mathcal{S}$.
By Assumption \ref{assump:boundedness}, it follows that
\[
\begin{aligned}
D^{-1}(t,x)\sum_{j\neq i}q_{ij} &\int_t^T e^{-(q_i+r(i))(u-t)}\left[\int_\mathbb{R} D(u,y)f_i(u,y;t,x) \, dy\right]du \\
& = D^{-1}(t,x)\sum_{j\neq i}q_{ij} \int_t^T e^{-(q_i+r(i))(u-t)}\, \mathbb{E}_{t,x,i}[D(u,X^{(i)}_u)]du \\
& \leq  \sum_{j\neq i}q_{ij} \int_t^T e^{-(q_i+r(i))(u-t)}\, e^{r(i)(u-t)}du \\
& = \sum_{j\neq i}q_{ij}\int_t^T e^{-q_i(u-t)}du = (1-e^{-q_i(T-t)})<1.
\end{aligned}
\]
Set $\rho=\max_{i\in \mathcal{M}} \{1-e^{-q_i(T-t)}\}$.
Upon applying the supremum norm $||\mathcal{T}(h)|| \leq   \rho ||h||$ it is implied that $\mathcal{T}$ is a contraction.
By the fixed point theorem on Banach spaces the equation \eqref{eq:inteq} has a fixed point, which implies that $H$ is the only solution.
\end{proof}

The significance of this result is that an iterative procedure can be furnished by means of the sequence of successive approximations (Picard iteration) 
\[
H_{n+1}= \mathcal{T}(H_n)+H_0, \qquad n=0,1,2, \ldots
\]
By the contraction principle, the sequence $\{H_n\}_{n\geq 0}$ converges to the unique solution $H$ of \eqref{eq:inteq}, with initial point $H_0$ as in \eqref{eq:inteq_initial} and the error estimate is
\[
|| H_n - H|| \leq \frac{\rho^n}{1-\rho} || H_1-H_0||
\]
for every $n>1$. 
In future work, we seek to implement the Picard iteration to approximate the solution and compare its performance and accuracy with existing PDE and Monte Carlo methods.

\section{Example with unbounded $\varphi$ and regime-switching Ornstein-Uhlenbeck model} \label{sec:eg}

In this section we provide an example in the context of derivatives pricing where the underlying spot price is $S_t=e^{X_t}$, and
the log-price process $X=(X_t)_{t\geq 0}$ follows the regime-switching Ornstein-Uhlenbeck dynamics
\begin{equation} \label{eq:OU}
dX_t=\beta( \theta(\alpha_t)-X_t)\,dt+\sigma(\alpha_t)dW_t.
\end{equation}
Both the long-run mean log price $\theta$, and the volatility $\sigma$, depend on the regime and are related by the expression
\begin{equation} \label{eq:tech_assump}
\theta(i)=r(i)-\frac{\sigma^2(i)}{2\beta},
\end{equation}
where we also allow the interest rate $r(i)>0$ to depend on the regime.
Evidence from futures prices in commodities such as oil, metals, agricultural products, and electricity
show that prices revert towards a mean reversion level (see e.g. \cite{Bessem_etal}).
The model in \eqref{eq:OU}-\eqref{eq:tech_assump} is an extension of the one-factor model described in Schwartz \cite{Schwartz} for commodity prices, which assumes constant parameters.
Here, the regimes represent the business cycle.
Empirical studies showing the effect of the economic regime in commodity spot prices can be seen in 
\cite[Ch. 22]{Commodities}) and \cite{Fama&French}.

In this context, we assume that the speed of mean reversion $\beta$ satisfies the condition $0<\beta\leq 1$,
which is supported by Schwartz's empirical studies \cite{Schwartz}.
For our example, consider a commodity call option with strike $K>0$ and maturity $T>0$ given by
\begin{equation} \label{eq:CommodityOption}
v(t,x,i)=\mathbb{E}_{t,x,i}\left[ e^{-\int_{t}^T r(\alpha_u)du}\varphi(X_T)\right],
\end{equation}
where $\varphi(x)= \max(e^{x}-K,0)$ is unbounded. It is well-known that the solution to \eqref{eq:F-P}, the density of the Ornstein-Uhlenbeck process, is
\[
f_i(u,y; t,x)=\frac{1}{\sqrt{2\pi \nu^2(u-t,i)}}\exp\left\{-\frac{[y- xe^{-\beta(u-t)}-m(u-t,i)]^2}{2\nu^2(u-t,i)} \right\}
\]
where
\begin{equation*}
m(s,i) :=\theta(i)(1-e^{-\beta s}), \qquad
\nu^2(s,i):=\frac{\sigma^2(i)}{2\beta}(1-e^{-2\beta s}).
\end{equation*}

We now define the dampening function and show that it satisfies the conditions in Assumption \ref{assump:boundedness}.
Consider the continuous function $D:[0,T]\times \mathbb{R}\mapsto \mathbb{R}\backslash\{0\}$ defined as
\begin{equation}\label{eq:D}
D(t,x):=\exp\{xe^{-\beta(T-t)}\}.
\end{equation}

Let us first show that for each $i\in \mathcal{M}$, the process $\{e^{-r(i)t}D(t,X_t^{(i)})\}_{0\leq t\leq T}$ is a supermartingale.
For $0\leq t\leq u\leq T$,
\[
\begin{split}
\mathbb{E}[D(u,X_u^{(i)})\mid \mathcal{F}_t] & = \mathbb{E}[\exp\{X_u^{(i)} e^{-\beta(T-u)}\}\mid X^{(i)}_t]  \\
            & = \exp\left\{ e^{-\beta(T-u)}\left(X_t^{(i)}e^{-\beta(u-t)}+m(u-t,i)\right)+\frac{1}{2}e^{-2\beta(T-u)}\nu^2(u-t,i)\right\} \\
            & \leq  D(t,X^{(i)}_t)\exp\left\{ e^{-\beta(T-u)}|m(u-t,i)|+\frac{1}{2}e^{-2\beta(T-u)}\nu^2(u-t,i)\right\}.
\end{split}
\]
By the definition of $m$ and $\nu$ along with the assumption in \eqref{eq:tech_assump} and $0<\beta\leq 1$,
\[
e^{-\beta(T-u)}|m(u-t,i)|+\frac{e^{-2\beta(T-u)}}{2}v^2(u-t,i)
\leq\left( \beta \theta(i)+\frac{1}{2}\sigma^2(i)\right)(u-t)
\leq r(i)(u-t).
\]
Therefore, $\mathbb{E}[D(u,X_u^{(i)})\mid \mathcal{F}_t]\leq D(t,X_t^{(i)})e^{r(i)(u-t)}$, as required.

Next, observe that $v(t,x,i) \leq \mathbb{E}_{t,x,i}[e^{X_T}]$ and so
\[
\begin{split}
\frac{v(t,x,i)}{D(t,x)} & \leq \mathbb{E}_{t,x,i}\left[\,\exp\left\{\int_t^T \theta(\alpha_u)e^{-\beta(T-u)}du+\frac{1}{2}\int_t^T \sigma^2(\alpha_u)e^{-2\beta(T-u)}du \right\}  \right]\\
                &\leq  \,\exp\left\{\sum_{i\in\mathcal{M}}\left(\int_t^T |\theta(i)|e^{-\beta(T-u)}du+\frac{1}{2}\int_t^T \sigma^2(i)e^{-2\beta(T-u)}du\right) \right\}  \\
                &\leq \prod_{i\in \mathcal{M}}e^{|m(T,i)|+\frac{1}{2}\nu^2(T,i)},
\end{split}
\]
where we used that $|m(\cdot,i)|$ and $\nu^2(\cdot,i)$ are increasing functions.
Thus, the \textit{dampened}  commodity call option $v(t,x,i)/D(t,x) \in \mathcal{S}$, and by
Theorem \ref{thm:main} it satisfies the integral equation \eqref{eq:inteq} where
$v_0(t,x,i)=\mathbb{E}_{t,x,i}[e^{-r(i)(T-t)}\varphi(X^{(i)}_T) ]$ is given by
\[
v_0(t,x,i)=e^{-r(i)(T-t)}\left[\,\exp\left\{xe^{-\beta(T-t)}+m(T-t,i)+\frac{1}{2}\nu^2(T-t,i) \right\} \Phi(d_1)-K\,\Phi(d_2)\right],
\]
with
\[
\begin{split}
d_1=d_1(t,x,i) &=\frac{xe^{-\beta(T-t)}-\ln(K)+m(T-t,i)+\nu^2(T-t,i)}{\nu(T-t,i)}, \\
d_2=d_2(t,x,i) & =d_1(t,x,i)-\nu(T-t,i),
\end{split}
\]
and $\Phi(\cdot)$ denotes the standard normal distribution.


\end{document}